\documentclass{amsart}

\newtheorem{theorem}{Theorem}[section]

\theoremstyle{definition}

\theoremstyle{remark}
\newtheorem{remark}[theorem]{Remark}

\numberwithin{equation}{section}

\newcommand{\hyper}[5]{\,{}_{#1}F_{#2}\left(
\begin{array}{r|l}
\begin{array}{cc}{\displaystyle{#3}}\\
{\displaystyle{#4}}
\end{array} & {\displaystyle{#5}}
\end{array} \right)}

\newcommand{\phin}[6]{\,{}S_{#1}^{*}\left(\!\!%
\begin{array}{c|c}
\begin{array}{cc}{\displaystyle{#2}} & {\displaystyle{#3}}\\[-0.1ex]
{\displaystyle{#4}} & {\displaystyle{#5}} \end{array} &
\,{\displaystyle{#6}}
\end{array} \right)}

\newcommand{\qphin}[6]{\,{}S_{#1}\left(\!\!%
\begin{array}{c|c}
\begin{array}{cc}{\displaystyle{#2}} & {\displaystyle{#3}}\\[-0.1ex]
{\displaystyle{#4}} & {\displaystyle{#5}} \end{array} & \, q;
\,{\displaystyle{#6}}
\end{array} \right)}

\newcommand{\qphinmonic}[6]{\,{}\bar{S}_{#1}\left(\!\!%
\begin{array}{c|c}
\begin{array}{cc}{\displaystyle{#2}} & {\displaystyle{#3}}\\[-0.1ex]
{\displaystyle{#4}} & {\displaystyle{#5}} \end{array} & \, q;
\,{\displaystyle{#6}}
\end{array} \right)}

\newcommand{\ww}[6]{\,{}W_{#1}\left(\!\!%
\begin{array}{c|c}
\begin{array}{cc}{\displaystyle{#2}} & {\displaystyle{#3}}\\[-0.1ex]
{\displaystyle{#4}} & {\displaystyle{#5}} \end{array} &
\,{\displaystyle{#6}}
\end{array} \right)}

\newcommand{\ssstar}[6]{\,{}S_{#1}^{*}\left(\!\!%
\begin{array}{c|c}
\begin{array}{cc}{\displaystyle{#2}} & {\displaystyle{#3}}\\[-0.1ex]
{\displaystyle{#4}} & {\displaystyle{#5}} \end{array} &
\,{\displaystyle{#6}}
\end{array} \right)}

\newcommand{\Z}{\mathbf Z}
\newcommand{\Zplus}{\Z_+}
\newcommand{\C}{\mathbf C}

\newcommand{\vectj}[2]{\mbox{$#1_1,#1_2,\ldots ,#1_{#2}$}}  

\newcommand{\qf}[3]{\text{$(#1;\,#2)_{#3}$}}
\newcommand{\qnum}[1]{\text{$[#1]_{q}$}}

\newcommand{\qbinom}[2]{\text{$\genfrac[]{0pt}{}{#1}{#2}_q$}}
\newcommand{\qqbinom}[2]{\text{$\genfrac[]{0pt}{}{#1}{#2}_{q^{2}}$}}

\newcommand{\qhyper}[6]{\,{}_{#1}\phi_{#2}\left(\!\!%
            \begin{array}{cc}{#3}\\[-0.1ex]{#4} \end{array}
            \Big|\,{#5};{#6}\right)}

\begin{document}

\title[A class of symmetric $q$-orthogonal polynomials]{A class of symmetric $q$-orthogonal polynomials with four free parameters}

\author[Area]{\sc I. Area}
\address[Area]{Departamento de Matem\'atica Aplicada II,
              E.E. de Telecomunicaci\'on,
              Universidade de Vigo,
              Campus Lagoas-Marcosende,
              36310 Vigo, Spain.}
\email[Area]{area@uvigo.es}
\thanks{The work of I. Area has been partially supported by the Ministerio de Ciencia e Innovaci\'on of Spain under grants MTM2009--14668--C02--01 and MTM2012--38794--C02--01, co-financed by the European Community fund FEDER.}

\author[Masjed-Jamei]{\sc M. Masjed-Jamei}
\address[Masjed-Jamei]{Department of Mathematics, K.N.Toosi University of Technology, P.O. Box 16315--1618, Tehran, Iran.}
\email[Masjed-Jamei]{mmjamei@kntu.ac.ir, mmjamei@yahoo.com}
\thanks{The work of M. Masjed-Jamei has been supported by a grant from``Iran National Science Foundation".}

\date{\today}

\begin{abstract}
By using a generalization of Sturm-Liouville problems in $q$-dif\-ference spaces, a class of symmetric $q$-orthogonal polynomials with four free parameters is introduced. The standard properties of these polynomials, such as a second order $q$-difference equation, the explicit form of the polynomials in terms of basic hypergeometric series, a three term recurrence relation and a general orthogonality relation are presented. Some particular examples are then studied in detail.
\end{abstract}


\maketitle

\section{Introduction}

A regular Sturm-Liouville problem of continuous type is a boundary value problem in the form
\begin{equation}\label{eq:1}
\frac{d}{dx} \left( k(x) \frac{dy_{n}(x)}{dx} \right) + \left(\lambda_{n} \varrho(x)-q(x) \right) y_{n}(x)=0 \qquad (k(x)>0, \varrho(x)>0),
\end{equation}
which is defined on an open interval, say $(a,b)$, and has the boundary conditions
\begin{equation}\label{eq:2}
\alpha_{1} y(a) + \beta_{1} y'(a)=0, \quad \alpha_{2} y(b) + \beta_{2} y'(b)=0,
\end{equation}
where $\alpha_{1}, \alpha_{2}$ and $\beta_{1}, \beta_{2}$, are given constants and $k(x)$, $k'(x)$, $q(x)$, and $\varrho(x)$ in (\ref{eq:1}) are to be assumed continuous for $x \in [a,b]$. In this sense, if one of the boundary points $a$ and $b$ is singular (i.e. $k(a) = 0$ or $k(b) = 0$), the problem is called a singular {S}turm-{L}iouville problem of continuous type.

Let $y_{n}$ and $y_{m}$ be two eigenfunctions of equation (\ref{eq:1}). According to {S}turm-{L}iouville theory \cite{MR922041}, they are orthogonal with respect to the weight function $\varrho(x)$ under the given conditions (\ref{eq:2}) so that we have
\begin{equation}\label{eq:3}
\int_{a}^{b} \varrho(x) y_{n}(x) y_{m}(x) dx =
\left( \int_{a}^{b} \varrho(x) y_{n}^{2}(x) dx \right) \delta_{n,m}=
\| y_{n} \|_{2}^{2} \begin{cases} 0 & n \neq m, \\ 1 & n=m. \end{cases}
\end{equation}

Many important special functions in theoretical and mathematical physics are solutions of a regular or singular {S}turm-{L}iouville problem that satisfy the orthogonality condition (\ref{eq:3}). For instance, the associated Legendre functions \cite{MR1810939}, Bessel functions \cite{MR922041}, Fourier trigonometric sequences \cite{MR0105586}, ultraspherical functions \cite{MR922041} and Hermite functions \cite{MR922041} are some specific continuous samples. Most of these functions are symmetric (i.e. $\phi_{n}(-x)=(-1)^{n} \phi_{n}(x)$) and have found valuable applications in physics and engineering. Hence, if we can somehow extend these examples symmetrically and preserve their orthogonality property, it seems that we will be able to find new applications, which logically extend the previous established applications. Recently in \cite{MR2374588}, this matter has been done for continuous variables and the classical equation (\ref{eq:1}) has been symmetrically extended in the following form
\begin{equation}\label{eq:4}
A(x) \phi_{n}''(x) + B(x) \phi_{n}'(x) + \left( \lambda_{n} C(x) + D(x) + \sigma_{n} E(x)\right) \phi_{n}(x)=0,
\end{equation}
where $A(x)$, $D(x)$, $E(x)$ and $(C(x)>0)$ are even functions, $B(x)$ is an odd function and 
\begin{equation}\label{eq:sigma}
\sigma_{n}=\frac{1-(-1)^{n}}{2} = \begin{cases}
0 & n \text{ even}, \\
1 & n \text{ odd}.
\end{cases} 
\end{equation}
It has been proved in \cite{MR2374588} that under some specific conditions, the symmetric solutions of equation (\ref{eq:4}) are orthogonal and preserve the orthogonality interval, in other words:
\begin{theorem}\cite{MR2374588}\label{th:1}
Let $\phi_{n}(-x)=(-1)^{n} \phi_{n}(x)$ be a sequence of symmetric functions that satisfies the differential equation (\ref{eq:4}), where $\{Ê\lambda_{n} \}_{n}$ is a sequence of constants. If $A(x)$, $(C(x) > 0)$, $D(x)$ and $E(x)$ are even real  functions and $B(x)$ is odd then
\begin{equation*}
\int_{-\nu}^{\nu} P^{*}(x) \phi_{n}(x) \phi_{m}(x) dx = \left( \int_{-\nu}^{\nu}P^{*}(x) \phi_{n}^{2}(x) dx \right) \delta_{n,m},
\end{equation*}
where
\begin{equation}\label{eq:6}
P^{*}(x)= C(x) {\exp} \left( \int \frac{B(x)-A'(x)}{A(x)} dx \right) = \frac{C(x)}{A(x)} {\exp} \left( \int \frac{B(x)}{A(x)} dx \right).
\end{equation}
Of course, the weight function defined in (\ref{eq:6}) must be positive and even on $[-\nu,\nu]$ and the function
\begin{equation*}
A(x)K(x)=A(x) {\exp} \left( \int \frac{B(x)-A'(x)}{A(x)} dx\right) = {\exp} \left( \int \frac{B(x)}{A(x)} dx\right),
\end{equation*}
must vanish at $x = \nu$ , i.e. $A(\nu) K(\nu) = 0$. In this way, since $K(x) = P^{*}(x) /C(x)$ is an even function so $A(-\nu) K(-\nu) = 0$ automatically.
\end{theorem}

By using theorem \ref{th:1}, many symmetric special functions of continuous type have been generalized in \cite{MR2270049,MR2374588,MR2421844,MR2601301,MR2467682,MR2743534,1220.33011}. Recently in \cite{MASJEDAREA} we have generalized usual Sturm-Liouville problems with symmetric solutions in discrete spaces on the linear lattice $x(s)=s$, and introduced a basic class of symmetric orthogonal polynomials of a discrete variable with four free parameters \cite{MASJEDAREA2}.

Moreover, $q$-orthogonal functions can similarly be solutions of a regular or singular $q$-Sturm-Liouville problem in the form \cite{0817.39004}
\begin{equation}\label{eq:8}
D_{q} \left( K^{*}(x) D_{q^{-1}} y_{n}(x) \right) + \left( \lambda_{n} \varrho^{*}(x)-q^{*}(x) \right)y_{n}(x)=0 \qquad (K^{*}(x)>0, \quad \varrho^{*}(x)>0),
\end{equation}
where the $q$-difference operator $D_{q}$ is defined by
\begin{equation}\label{eq:qhahn}
D_{q} f(x)=\frac{f(qx)-f(x)}{(q-1)x}\qquad (x\neq 0),
\end{equation}
with $D_{q} f(0):=f'(0)$ (provided $f'(0)$ exists), and (\ref{eq:8}) satisfies a set of boundary conditions like (\ref{eq:2}). This means that if $y_{n}(x)$ and $y_{m}(x)$ are two eigenfunctions of $q$-difference equation (\ref{eq:8}), they are orthogonal with respect to the weight function $\varrho^{*}(x)$ on a discrete set \cite{MR1149380}.

Recently in \cite{AREAMASJED} we have presented the following theorem by which one can generalize usual Sturm-Liouville problems with symmetric solutions. As a very important consequence of this theorem, we can introduce a basic class of symmetric $q$-orthogonal polynomials with four free parameters.

\begin{theorem}
Let $\phi_{n}(x;q)=(-1)^{n} \phi_{n}(-x;q)$ be a sequence of symmetric functions that satisfies the $q$-difference equation 
\begin{equation}\label{eq:17}
A(x) D_{q} D_{q^{-1}} \phi_{n}(x;q) + B(x) D_{q} \phi_{n}(x;q) \\ 
+ \left( \lambda_{n,q} C(x) + D(x) + \sigma_{n} E(x) \right) \phi_{n}(x;q)=0,
\end{equation}
where $A(x)$, $B(x)$, $C(x)$, $D(x)$ and $E(x)$ are independent functions, $\sigma_{n}$ is defined in (\ref{eq:sigma})
and $\lambda_{n,q}$ is a sequence of constants. If $A(x)$, $(C(x)>0)$, $D(x)$ and $E(x)$ are even functions and $B(x)$ is odd, then
\begin{equation*}
\int_{-\alpha}^{\alpha} W^{*}(x;q) \phi_{n}(x;q) \phi_{m}(x;q) d_{q}x =
\left( \int_{-\alpha}^{\alpha} W^{*}(x;q) \phi_{n}^{2}(x;q) d_{q}x  \right) \delta_{n,m},
\end{equation*}
where 
\begin{equation}\label{eq:19}
W^{*}(x;q)=C(x)W(x;q),
\end{equation}
and $W(x;q)$ is solution of the Pearson $q$-difference equation
\begin{equation}\label{eq:20}
D_{q} \left(A(x) W(x;q) \right) = B(x) W(x;q),
\end{equation}
which is equivalent to
\begin{equation*}
\frac{W(qx;q)}{W(x;q)}=\frac{(q-1)x B(x)+A(x)}{A(qx)}.
\end{equation*}
Of course, the weight function defined in (\ref{eq:19}) must be positive and even and $A(x)W(x;q)$ must vanish at $x=\alpha$.
\end{theorem}

\section{Basic definitions and notations}

The $q$-shifted factorial is defined by
\begin{equation*}
 \qf xqn=\prod_{j=0}^{n-1}(1-q^jx),\qquad n=0,1,\ldots,
\end{equation*}
the $q$-number by
\begin{equation}\label{eq:qnum}
\qnum{z}=\frac{q^{z}-1}{q-1}, \quad z \in {\mathbf{C}},
\end{equation}
and the {basic hypergeometric series} is defined by
\begin{equation*}
 \qhyper rs {\vectj ar}{\vectj bs}{q}{z}=
 \sum_{k=0}^{\infty}\frac{\qf{a_1}qk\ldots \qf{a_r}qk}
                                 {\qf qqk\qf{b_1}qk\ldots \qf{b_s}qk}
               \left((-1)^kq^{\binom{k}{2}}\right)^{1+s-r}z^k.
\end{equation*}
Here $r,\,s\in\Zplus$ and $\vectj ar$, $\vectj bs$, $z$ $\in \C$.
In order to have a well--defined series the condition $ \vectj
bs \neq q^{-k}$ ($k=0,1,\ldots $) is required.

Let $\mu \in {\mathbf{C}}$ be fixed. A set $A \subseteq {\mathbf{C}}$ is called a $\mu$-geometric set if for $x \in A$, $\mu x \in A$. Let $f$ be a function defined on a $q$-geometric set $A \subseteq {\mathbf{C}}$. The $q$-difference operator is defined by
\begin{equation*}
D_{q}f(x)=\frac{f(qx)-f(x)}{(q-1)x}, \qquad x \in A \setminus \{0\}.
\end{equation*}
If $0 \in A$, we say that $f$ has the $q$-derivative at zero if the limit
\begin{equation*}
\lim_{n \to \infty} \frac{f(x q^{n})-f(0)}{xq^{n}} \qquad (x \in A),
\end{equation*}
exists and does not depend on $x$. We then denote this limit by $D_{q} f(0)$.

We shall also need the $q$-integral (the inverse of the $q$-derivative operator) introduced by J. Thomae \cite{Thomae1969} and F.H. Jackson \cite{JACKSON1910} ---see also \cite{MR2128719,MR2191786,MR2656096}--- which is defined as
\begin{equation}\label{eq:11}
\int_{0}^{x} f(t) \text{d}_qt = x(1-q)\sum_{n=0}^{\infty} q^n f(q^n x)\,, \qquad (x \in A),
\end{equation}
provided that the series converges, and for the interval $[a,b]$ we have based on (\ref{eq:11}) that
\begin{equation}\label{eq:10}
\int_{a}^{b} f(t) d_{q}t = \int_{0}^{b} f(t) d_{q}t - \int_{0}^{a} f(t) d_qt, \qquad (a,b \in A).
\end{equation}
Relations (\ref{eq:11}) and (\ref{eq:10}) directly yield
\begin{equation}\label{eq:12}
\int_{-b}^{b} f(t) d_{q}t = b(1-q) \sum_{n=0}^{\infty} q^{n} \left( f(bq^{n})+f(-b q^{n}) \right), \qquad (b \in A).
\end{equation}
This means that if $f$ is an odd function, then  $\displaystyle{\int_{-b}^{b} f(t) d_{q}t=0}$. Moreover, if $b \to \infty$, (\ref{eq:12}) changes to
\begin{equation*}
\int_{-\infty}^{\infty} f(t) d_{q}t =(1-q) \sum_{n=-\infty}^{\infty} q^{n} \left( f(q^{n}) + f(-q^{n})\right).
\end{equation*}

A function $f$ which is defined on a $q$-geometric set $A$ with $0 \in A$ is said to be $q$-regular at zero if $\displaystyle{\lim_{n \to \infty} f(xq^{n})=f(0)}$ for every $x \in A$. The rule of $q$-integration by parts is denoted by
\begin{equation}\label{eq:14}
\int_{0}^{a} g(x) D_{q}f(x) d_{q}x=(fg)(a)-\lim_{n \to \infty} (fg)(aq^{n}) - \int_{0}^{a} D_{q} g(x)f(qx)d_{q}x.
\end{equation}
If $f,g$ are $q$-regular at zero, the $\displaystyle{\lim_{n \to \infty} (fg)(aq^{n})}$ on the right-hand side of (\ref{eq:14}) can be replaced by $(fg)(0)$.

For $0 < R \leq \infty$ let $\Omega_{R}$ denote the disc $\{ z \in {\mathbf{C}} \, : \, \vert z \vert < R\}$. The $q$-analogue of the fundamental theorem says:
Let $f: {\Omega_{R}} \to {\mathbf{C}}$ be $q$-regular at zero and $\theta \in \Omega_{R}$ be fixed. Define
\begin{equation*}
F(x)=\int_{\theta}^{x} f(t) d_{q}t \qquad (x \in \Omega_{R}).
\end{equation*}
Then, the function $F$ is $q$-regular at zero, $D_{q}F(x)$ exists for any $x \in \Omega_{R}$ and $D_{q}F(x)=f(x)$. Conversely, if $a,b \in \Omega_{R}$ then
\begin{equation*}
\int_{a}^{b} D_{q}f(t) d_{q}t=f(b)-f(a).
\end{equation*}

The function $f$ is $q$-integrable on $\Omega_{R}$ if $\displaystyle{ \vert f(t) \vert d_{q}t}$ exists for all $x \in \Omega_{R}$.

\section{A class of symmetric $q$-orthogonal polynomials}

As a special case of equation (\ref{eq:4}), the following differential equation is defined in \cite{MR2270049}:
\begin{equation}\label{eq:240}
x^{2}(ax^{2}+b) \Phi_{n}''(x) + x(cx^{2}+d) \Phi_{n}'(x)-(n(c+(n-1)a)x^{2}+\sigma_{n} d) \Phi_{n}(x)=0.
\end{equation}
This equation has a symmetric polynomial solution as
\begin{multline}\label{eq:solcont}
\ssstar{n}{{c}}{{d} }{{a}}{{b}}{t}\\=\sum_{k=0}^{[n/2]} \binom{[n/2]}{k} \left( \prod_{i=0}^{[n/2]-(k+1)} \frac{(2i+(-1)^{n+1}+2[n/2]){a}+{c}}{(2i+(-1)^{n+1}+2){b}+{d} } \right) x^{n-2k} .
\end{multline}
If $a,b \neq0$, (\ref{eq:solcont}) can be written in terms of a $_{2}F_{1}$ hypergeometric series as
\begin{multline*}
\ssstar{n}{{c}}{{d} }{{a}}{{b}}{x} \\=\left(\frac{{a} }{{b}}\right)^{ \left[{n}/{2}\right]}
\frac{ \Gamma \left(\frac{1}{2} \left(1+\frac{{d}}{{b}} \right)+\sigma_{n}\right)  \Gamma
   \left(\frac{1}{2} \left(\frac{{c} }{{a} }-1\right)+2
    \left[{n}/{2}\right]+\sigma_{n}\right)}
   {\Gamma \left(\frac{1}{2} \left(\frac{{c}}{{a}}-1\right)+ \left[{n}/{2}\right]+\sigma_{n}\right) 
   \Gamma \left(\frac{1}{2} \left(1+\frac{{d}}{{b}} \right)+ \left[{n}/{2}\right]+\sigma_{n}\right)} \\
   \times \,\,x^{n}\,\, 
   \hyper{2}{1}
   {- \left[{n}/{2}\right],\frac{-2 \left( \left[{n}/{2}\right]+\sigma_{n}\right) {b} +{b} -{d} }{2 {b} }}
   {-\frac{{c} }{2{a} }-n+\frac{3}{2}}{-\frac{{b} }{{a} x^2}}.
\end{multline*}

The weight function corresponding to the polynomials (\ref{eq:solcont}) is in the form
\begin{multline}
\ww{}{c}{d}{a}{b}{x}=x^{2} \text{exp} \left( \int \frac{(c-4a)x^{2}+(d-2b)}{x(ax^2+b)} dx \right) \\
= \text{exp} \left( \int \frac{(c-2a)x^2+d}{x(ax^2+b)}dx \right).
\end{multline}
which satisfies the equation \cite{MR2270049}
\begin{equation}
\frac{d}{dx} (x^2(ax^2+b)W(x))=x((c+2a)x^2+(d+2b))W(x),
\end{equation}
For instance we have
\begin{align*}
&K_{1} \ww{}{-2a-2b-2}{2a}{-1}{1}{x}=\frac{\Gamma(a+b+3/2)}{\Gamma(a+1/2)\Gamma(b+1)} x^{2a}(1-x^{2})^{b}, \\
&\text{\hspace*{5cm}} \quad -1 \leq x \leq 1; \quad a+1/2>0, \quad b+1>0,
\end{align*}
and
\begin{equation*}
K_{2} \ww{}{-2}{2a}{0}{1}{x}=\frac{1}{\Gamma(a+1/2)} x^{2a}e^{-x^{2}}, 
\quad  x \in (-\infty, \infty); \quad a+1/2>0.
\end{equation*}
The values $K_{1}$ and $K_{2}$ play the normalizing constant role in the above distributions.
By referring to theorem \ref{th:1}, we observe in (\ref{eq:240}) that $A(x)=x^{2}({a} x^{2}+{b})$ is a polynomial of degree at most four, $B(x)=x({c} x^{2}+{d} )$ is an odd polynomial of degree at most three, $C(x)=x^{2}$ is a symmetric quadratic polynomial, $D(x) = 0$, and $E(x) =- {d}$ is constant. 

Motivated by these options, in this paper we similarly consider a $q$-difference equation type of (\ref{eq:17}) as 
\begin{equation}\label{eq:qde}
x^2 \left({a}  x^2+{b} \right) D_{q} D_{q^{-1}} \phi_{n}(x;q) + x \left({c}  x^2+{d} \right) D_{q} \phi_{n}(x;q) \\ 
+ \left( \lambda_{n,q} x^{2} - \sigma_{n} d \right) \phi_{n}(x;q)=0.
\end{equation}

To find a symmetric $q$-orthogonal polynomial solution of (\ref{eq:qde}), let
\begin{equation}\label{eq:29}
\bar{\phi}_{n}(x;q)=x^{n}+\delta_{n,q} x^{n-2} + \cdots,
\end{equation}
satisfy a three term recurrence relation as
\begin{equation}\label{eq:30}
\bar{\phi}_{n+1}(x;q)=x \bar{\phi}_{n}(x;q)-C_{n,q} \bar{\phi}_{n-1}(x;q), \quad ({\rm{with }}\,\,\, \bar{\phi}_{0}(x;q)=1, \,\,\, \bar{\phi}_{1}(x;q)=x).
\end{equation}

From (\ref{eq:qde}) and (\ref{eq:29}), equating the coefficient in $x^{n+2}$ gives
\begin{equation}\label{eq:lambdan}
\lambda_{n,q}=-\qnum{n} \left(c-\qnum{1-n}a \right),
\end{equation}
provided that $\vert {a} \vert + \vert {c} \vert \neq 0$.

By using the eigenvalue $\lambda_{n,q}$ given in (\ref{eq:lambdan}) and equating the coefficient in $x^{n}$ we obtain in (\ref{eq:29}) that
\begin{equation*}
\delta_{n,q}=\frac{q^2 \left(-{b}  (q-1) q \qnum{n-1}\qnum{n}-{d}  q^{2 n}+{d}  q^n (q \sigma_{n}
+\sigma_{n-1})\right)}{(q+1) \left({a}  q^3-q^{2 n} ({a} +{c}  (q-1))\right)}.
\end{equation*}

Also from (\ref{eq:29}) and (\ref{eq:30}) we have
\begin{multline*}
x^{n+1} +  \delta_{n+1,q} x^{n-1} + \cdots \\= x \left( x^{n} +  \delta_{n,q} x^{n-2} + \cdots  \right) - {C}_{n,q} \left( x^{n-1} +  \delta_{n-1,q} x^{n-3}  + \cdots \right),
\end{multline*}
which implies
\begin{multline}\label{eq:27new}
C_{n,q}=\delta_{n,q}-\delta_{n+1,q}=\frac{1}{{a} ^2 q^4+q^{4 n} ({a} +{c}  (q-1))^2-{a} 
   \left(q^3+q\right) q^{2 n} ({a} +{c}  (q-1))} \\ \times \left( q^{n+1} \left(q^{2 n} ({a} +{c}  (q-1)) (({d} -{d}  q)
   \sigma_{n}-{b} ) \right. \right. \\ \left. \left. +q^n \left({a}  \left({b} 
   \left(q^2+1\right)+{d}  (q-1) q^2\right)+{b}  {c} 
   (q-1)\right)-{a}  q^2 ({b} +{d}  (q-1) \sigma_{n-1})\right) \right).
\end{multline}

\begin{remark}
The limit case of (\ref{eq:27new}) is as
\begin{equation}
\lim_{q \uparrow 1} C_{n,q}=\frac{n ({a}  ({b}  (2-n) +{d} )-{b}  {c}
   )-{d}  \sigma_{n} (2 {a}  (n-1)+{c} )}{({a}  (2 n-3)+{c} ) ({a}  (2 n-1)+{c} )},
\end{equation}
and for the eigenvalue (\ref{eq:lambdan}) we have
\begin{equation}
\lim_{q \uparrow 1} \lambda_{n,q}=-n (c-(1-n)a),
\end{equation}
which are exactly the same as in the continuous case \cite{MR2421844} by taking into account that the three-term recurrence relation (\ref{eq:30}) has a minus sign in the coefficients $C_{n,q}$.
\end{remark}

Since the polynomial solution of equation (\ref{eq:qde}) is symmetric, we use the notation
\begin{equation*}
\phi_{n}(x;q)=\qphin{n}{c}{d}{a}{b}{x},
\end{equation*}
for mathematical formulae and $S_{n}(a,b,c,d,x;q)$ into the text. This means that from now we deal with just one characteristic vector  $\vec{V}=(a,b,c,d)$ for any given sub-case.

For $n=2m$ and $n=2m+1$, $C_{n,q}$ in (\ref{eq:27new}) are simplified as
\begin{equation*}
C_{2m,q}  = -\frac{q^{2 m+1} \qnum{2m}(q-1) \left({b}  q^{2 m}
   ({a} +{c}  (q-1))-{a}  q^2 ({b} +{d} 
   (q-1))\right)}{{a} ^2 q^4+q^{8 m} ({a} +{c} 
   (q-1))^2-{a}  \left(q^3+q\right) q^{4 m} ({a} +{c} 
   (q-1))}\,,
\end{equation*}
and
\begin{equation*}
C_{2m + 1,q}  = -\frac{q^{2 m} \left(q^{2 m} ({a} +{c}  (q-1))-{a} 
   q\right) \left(q^{2 m+1} ({b} +{d}  (q-1))-{b}
   \right)}{{a} ^2 q+q^{8 m+1} ({a} +{c} 
   (q-1))^2-{a}  \left(q^2+1\right) q^{4 m} ({a} +{c} 
   (q-1))}\,.
\end{equation*}

\begin{remark}\label{remark:positivity}
Once we have explicitly determined $C_{n,q}$ in the recurrence relation (\ref{eq:30}), a discussion about the situation of this coefficient is extremely important. For instance, analyzing the location of the zeros of orthogonal polynomials would give rise to a positive definite case when $C_{n,q}  > 0\,\,\,(\forall n \in {\mathbf{N}})$, the quasi-definite case when $C_{n,q} \neq 0$, and weak orthogonality case when $C_{n,q}=0$ for some values of $n$.
However, this discussion completely depends on the four parameters ${a},{b},{c}$ and ${d}$, because $C_{n,q}$ is in general a rational expression in $q^{n}$.
\end{remark}

\begin{theorem}\label{theorem:explicit}
The explicit form of the polynomial $S_{m}(a,b,c,d;x;q)$ is as
\begin{multline}\label{eq:exp}
\qphin{m}{c}{d}{a}{b}{x} \\ 
=\sum _{k=0}^{ \left[\frac{m}{2}\right]}
   q^{(k-1) k} x^{m-2 k}
   \qqbinom{ \left[\frac{m}{2}\right]}{k}
\prod_{j=0}^{\left[\frac{m}{2}\right]-k-1} 
\frac{a \qnum{2 j+\sigma_{m}+m-1}+c q^{2 j+\sigma_{m}+m-1}}
{b \qnum{\left(2 j+(-1)^{m+1}+2\right)}+d q^{2 j+(-1)^{m+1}+2}},
\end{multline}
where the $q$-number $\qnum{x}$ has been defined in (\ref{eq:qnum}) and the $q$-binomial coefficient is defined by
\[
\qbinom{n}{m}=\frac{\qf{q}{q}{n}}{\qf{q}{q}{m} \qf{q}{q}{n-m}}.
\]

Moreover, if $a,b \neq 0$, then
\begin{equation}\label{eq:explicit2m}
\qphin{m}{c}{d}{a}{b}{x} = x^{\sigma_{m}} \qhyper{2}{1}
{q^{-m+\sigma_{m}},\frac{({a} +{c}  (q-1)) q^{m+\sigma_{m}-1}}{{a} }}
{\frac{({b}   +{d}  (q-1)) q^{2 \sigma_{m}+1}}{{b} }}
{q^{2}}
{-\frac{{a} q^2 x^2}{{b} }}.
\end{equation}
\end{theorem}
\begin{proof}
Despite the degrees of $A(x)$, $B(x)$, $C(x)$ and $E(x)$ the proof can be done in a similar way as in \cite[Section 10.2]{MR2656096} for classical $q$-orthogonal polynomials.
\end{proof}

\begin{remark}
It is easy to check that
\[
\lim_{q \uparrow 1} \qphin{m}{c}{d}{a}{b}{x}=\phin{m}{c}{d}{a}{b}{x},
\]
where the right-hand side polynomial has been introduced in (\ref{eq:solcont}).
\end{remark}

\begin{remark}
The monic form of the polynomials (\ref{eq:explicit2m}) is represented as
\begin{multline*}
\qphinmonic{m}{c}{d}{a}{b}{x}\\=q^{\sigma_{m}-m}  \left( -\frac{b}{a} \right)^{ \left[\frac{m}{2}\right]}
\frac{
\qf{q^2;q^{-m \sigma_{m}};\frac{(b +d  (q-1)) q^{2 \sigma_{m}+1}}{b}}{q^{2}}{\left[\frac{m}{2}\right]}
}
{
\qf{q^{-m};q^{-(m-1) \sigma_{m}};\frac{(a +c  (q-1)) q^{m+\sigma_{m}-1}}{a}}{q^{2}}{\left[\frac{m}{2}\right]}
} \\
\times x^{\sigma_{m}} \qhyper{2}{1}
{q^{-m+\sigma_{m}},\frac{({a} +{c}  (q-1)) q^{m+\sigma_{m}-1}}{{a} }}
{\frac{({b}   +{d}  (q-1)) q^{2 \sigma_{m}+1}}{{b} }}
{q^{2}}
{-\frac{{a} q^2 x^2}{{b} }}.
\end{multline*}
\end{remark}

By noting theorem 2, since
\begin{equation}\label{eq:qpearson2}
\frac{W(qx;q)}{W(x;q)}=\frac{A(x)+(q-1) x B(x)}{A(qx)}
=\frac{x^2 ({a} +{c}  (q-1))+{b} +{d}  (q-1)}{q^{2}\left({a} q^2 x^2+{b} \right)},
\end{equation}
if $a,b \neq 0$ then the solution of equation (\ref{eq:qpearson2}) would be
\[
W(x;q)=\left(\frac{d (q-1)}{b}+1\right)^{\frac{\log x^2}{2 \log q}} 
\frac{
\qf{-\frac{a q^2 x^2}{b}}{q^2}{\infty}}
{x^{2}\, \qf{-\frac{(a+c (q-1)) x^2}{b+d (q-1)}}{q^2}{\infty}}=W(-x;q),
\]
in which some restrictions on the parameters must be considered in order to have convergence for the infinity products.

To compute the norm square value of the symmetric polynomials (\ref{eq:exp}), we can use Favard's theorem \cite{MR0481884}, which says if $\{P_{n}(x;q)\}$ satisfies the recurrence relation
\begin{equation*}
x P_{n}(x;q) = A_{n} P_{n+1}(x;q) + B_{n} P_{n}(x;q) + C_{n} P_{n-1}(x;q), \qquad n=0,1,2,\dots,
\end{equation*}
where $P_{-1}(x;q)=0$, $P_{0}(x;q)=1$, $A_{n}$, $B_{n}$, $C_{n}$ real and $A_{n} C_{n+1}>0$ for $n=0,1,2,\dots$, then there exists a weight function $W^{*}(x;q)$ so that
\begin{equation*}
\int_{-\alpha}^{\alpha} W^{*}(x;q) P_{n}(x;q) P_{m}(x;q) d_{q}x=
\left( \prod_{i=0}^{n-1} \frac{C_{i+1}}{A_{i}} \int_{-\alpha}^{\alpha} W^{*}(x;q) d_{q}x \right) \delta_{n,m}.
\end{equation*}
It is clear that this theorem also holds for the monic type of symmetric $q$-polynomials in which $A_{n}=1$ and $B_{n}=0$. 

We are now in conditions to analyze some particular cases of the $q$-difference equation (\ref{eq:qde}), which shall provide $q$-analogues of different families of orthogonal polynomials. 

\subsection{A generalization of $q$-ultraspherical polynomials} Let us consider the following $q$-difference equation
\begin{multline}\label{eq:gup}
{x^2 \left(1-x^2\right)} D_{q}D_{q^{-1}} \phi_{n}(x;q)
+(q+1) q x \left(\alpha-x^2 (\alpha+\beta+1)\right) D_{q} \phi_{n}(x;q) \\+
\left( -\qnum{n} (-(1+\alpha+\beta)q(1+q)+\qnum{1-n}) x^2 -\alpha q (1+q)\sigma_{n} \right) \phi_{n}(x;q)=0,
\end{multline}
as a special case of (\ref{eq:qde}) with the polynomial solution
\begin{multline}\label{eq:37}
\phi_{n}(x;\alpha,\beta \vert q)=\qphin{m}{-q (q+1) (\alpha+\beta+1)}{\alpha q (q+1)}{-1}{1}{x} \\ =x^{\sigma_{n}}
\qhyper{2}{1}
{q^{\sigma_{n}-n},q^{n+\sigma_{n}-1} \left((\alpha+\beta+1) q \left(q^2-1\right)+1\right)}
{q^{2\sigma_{n}+1} \left(\alpha q \left(q^2-1\right)+1\right)}{q^2}{q^2 x^2} .
\end{multline}
The polynomial sequence (\ref{eq:37}) satisfies an orthogonality relation as
\begin{multline*}
\int_{-1}^{1} W_{1}^{*}(x;\alpha,\beta \vert q) \phi_{n}(x;\alpha,\beta \vert q) \phi_{m}(x;\alpha,\beta \vert q) d_{q}x\\=
\left(\int_{-1}^{1} W_{1}^{*}(x;\alpha,\beta \vert q) \phi_{n}^{2}(x;\alpha,\beta \vert q) d_{q}x\right) \delta_{n,m},
\end{multline*}
in which $W_{1}^{*}(x;\alpha,\beta \vert q)=x^{2}\,W_{1}(x;\alpha,\beta \vert q)$ is the main weight function and the function $W_{1}(x;\alpha,\beta \vert q)$ satisfies the equation
\begin{equation}\label{eq:qde1}
\frac{W_{1}(qx;\alpha,\beta \vert q)}{W_{1}(x;\alpha,\beta \vert q)}=\frac{q \left(q^2-1\right) \left(x^2
   (\alpha+\beta+1)-\alpha\right)+x^2-1}{q^2 \left(q^2 x^2-1\right)}.
\end{equation}

Up to a periodic function, a solution of the equation (\ref{eq:qde1}) is in the form
\[
W_{1}(x;\alpha,\beta \vert q)=B^{\frac{\log\left(x^2\right)}{\log (q)}} \frac{\left(q^2 x^2;q^2\right){}_{\infty } }{x^{2} \, \left(-\frac{A^2
   x^2}{B^2};q^2\right){}_{\infty }}=W_{1}(-x;\alpha,\beta \vert q),
\]
where
\[
A=\sqrt{\left(q-q^3\right) (\alpha+\beta+1)-1}, \qquad B=\sqrt{\alpha \left(q^3-q\right)+1}.
\]
Notice that
\[
\lim_{q \uparrow 1} W_{1}^{*}(x;\alpha,\beta \vert q)=\lim_{q \uparrow 1} x^{2} W_{1}(x;\alpha,\beta \vert q)=x^{2\alpha}(1-x^{2})^{\beta},
\]
which gives us the orthogonality weight function of generalized ultra spherical polynomials, and might be compared with \cite[Eq. (24)]{MR2421844}.

The monic type of polynomials (\ref{eq:37}) satisfies a three term recurrence relation of type (\ref{eq:30}) with
\[
C_{2m,q}=\frac{q^{2 m+2} \left(q^{2 m}-1\right) \left(q^{2 m} \left(q
   \left(q^2-1\right) \vartheta+1\right)+q^2 \left(\alpha
   \left(q-q^3\right)-1\right)\right)}{-\left(q^2+1\right) q^{4
   m+2} \left(q \left(q^2-1\right) \vartheta+1\right)+q^{8 m+1}
   \left(q \left(q^2-1\right) \vartheta+1\right)^2+q^5},
\]
and
\begin{multline*}
C_{2m+1,q}\\=\frac{q^{2 m} \left(q^{2 m} \left(-\left(\alpha  q^5+(\beta +1)
   q^3+q^2-q {\vartheta}+1\right)\right)+\left(\alpha  q
   \left(q^2-1\right)+1\right) q^{4 m+1} \left(q
   \left(q^2-1\right)
   {\vartheta}+1\right)+q\right)}{-\left(q^2+1\right) q^{4
   m} \left(q \left(q^2-1\right) {\vartheta}+1\right)+q^{8
   m+1} \left(q \left(q^2-1\right)
   {\vartheta}+1\right)^2+q},
\end{multline*}
where $\vartheta=\alpha+\beta+1$.
Notice in this case that
\[
\lim_{q \uparrow 1} C_{n,q} = \frac{n^2-2 \left(\alpha (-1)^n (\alpha+\beta+n)-(\alpha+\beta) (\alpha+n)\right)}{(2 \alpha+2 \beta+2 n-1) (2 \alpha+2 \beta+2 n+1)},
\]
which coincides with \cite[Eq. (51.1)]{MR2270049}.

Hence, the norm square value of the monic type of the $q$-polynomials (\ref{eq:37}) takes the form
\begin{equation*}
\int_{-1}^{1} \bar{\phi}_{n}^{2}(x; \alpha,\beta \vert q) W_{1}^{*}(x;\alpha,\beta \vert q) d_{q}x= d_{n,\alpha,\beta}^{2} 
\int_{-1}^{1} W_{1}^{*}(x;\alpha,\beta \vert q) d_{q}x,
\end{equation*}
where
\begin{multline*}
d_{2m,\alpha,\beta}^{2}=\frac{\left(q^2;q^2\right)_m \left(q \left(\alpha  q
   \left(q^2-1\right)+1\right);q^2\right)_m}{\left(\frac{(\alpha +\beta +1) q
   \left(q^2-1\right)+1}{q^3};q^4\right)_{m+1}
   \left(\frac{(\alpha +\beta +1) q
   \left(q^2-1\right)+1}{q};q^4\right){}_m}  \\Ê\times
\frac{(q-1) q^{m (2 m-1)-2} (q (q+1) (\alpha +\beta )-1) (q
   (q+1) (\alpha +\beta +1)-1) \left(\alpha  q
   \left(q^2-1\right)+1\right)^{m+1}}{(q+1) \left(-q (\alpha
   +\beta +1)+\alpha  q^3+1\right)} \\
    \times
   \frac{\left(\frac{(\alpha +\beta +1) q
   \left(q^2-1\right)+1}{q};q^2\right)_m \left(\frac{(\alpha
   +\beta +1) q \left(q^2-1\right)+1}{q^2 \left(\alpha  q
   \left(q^2-1\right)+1\right)};q^2\right)_{m+1}}{\left(\frac
   {(\alpha +\beta +1) q
   \left(q^2-1\right)+1}{q};q^4\right)_{m+1} \left(q
   \left((\alpha +\beta +1) q
   \left(q^2-1\right)+1\right);q^4\right)_m},
\end{multline*}
and
\begin{multline*}
d_{2m+1,\alpha,\beta}^{2}=\frac{\left(q^2;q^2\right)_m \left(q \left(\alpha  q
   \left(q^2-1\right)+1\right);q^2\right)_{m+1}}{\left(q
   \left((\alpha +\beta +1) q
   \left(q^2-1\right)+1\right);q^4\right){}_{m+1}}\\ \times \frac{(q-1) q^{2 m^2+m-2} (q (q+1) (\alpha +\beta )-1) (q (q+1)
   (\alpha +\beta +1)-1) \left(\alpha  q
   \left(q^2-1\right)+1\right)^{m+1}}{(q+1) \left(-q (\alpha
   +\beta +1)+\alpha  q^3+1\right)} \\
   \times 
   \frac{\left(\frac{(\alpha +\beta +1) q
   \left(q^2-1\right)+1}{q};q^2\right)_{m+1}
   \left(\frac{(\alpha +\beta +1) q \left(q^2-1\right)+1}{q^2
   \left(\alpha  q
   \left(q^2-1\right)+1\right)};q^2\right)_{m+1}}{\left(\frac
   {(\alpha +\beta +1) q
   \left(q^2-1\right)+1}{q^3};q^4\right)_{m+1}
   \left(\left(\frac{(\alpha +\beta +1) q
   \left(q^2-1\right)+1}{q};q^4\right)_{m+1}\right)^2}.
\end{multline*}

\subsubsection{Fifth kind $q$-Chebyshev polynomials}

The $q$-difference equation
\begin{multline*}
x^{2}(1-x^{2}) D_{q} D_{q^{-1}} \phi_{n}(x;q)
+qx \left(-(q^{2}+q+1) x^2+q+1 \right) D_{q} \phi_{n}(x;q) 
\\+\left(\qnum{n} \left( q (q^{2}+q+1)-\qnum{1-n} \right) x^{2} -q (1+q) \sigma_{n} \right) \phi_{n}(x;q) =0,
\end{multline*}
is a special case of (\ref{eq:gup}) for $\alpha=1$ and $\beta=\qnum{3}/\qnum{2}-2$ with the polynomial solution
\begin{multline}\label{eq:icod}
\phi_{n}(x;1, \frac{\qnum{3}}{\qnum{2}}-2 \vert q)=\qphin{m}{-q \left(q^2+q+1\right)}{q (q+1)}{-1}{1}{x} \\
=x^{\sigma_{n}} \, \qhyper{2}{1}
{q^{\sigma_{n}-n},q^{n+\sigma_{n}-1} \left(q^4-q+1\right)}
{q^{2 \sigma_{n}+1} \left(q^3-q+1\right)}
{q^{2}}
{q^2 x^2},
\end{multline}
satisfying the orthogonality relation of monic type
\begin{multline*}
\int_{-1}^{1} W_{1}^{*}(x;1,\frac{\qnum{3}}{\qnum{2}}-2 \vert q) \bar{\phi}_{n}(x; 1,\frac{\qnum{3}}{\qnum{2}}-2 \vert q) \bar{\phi}_{m}(x; 1,\frac{\qnum{3}}{\qnum{2}}-2 \vert q) d_{q}x \\= 
d_{n,1,\frac{\qnum{3}}{\qnum{2}}-2}^{2}  \left( \int_{-1}^{1} W_{1}^{*}(x;1,\frac{\qnum{3}}{\qnum{2}}-2 \vert q) d_{q}x
\right) \delta_{n,m} ,
\end{multline*}
in which  
\begin{equation*}
W_{1}^{*}(x;1,\frac{\qnum{3}}{\qnum{2}}-2 \vert q)=\frac{
\qf{q^2 x^2}{q^{2}}{\infty}
   \left(q^3-q+1\right)^{\frac{\log \left(x^2\right)}{2 \log(q)}}}
 { \qf{\frac{q^{4}-q+1}{q^3-q+1}x^2}{q^{2}}{\infty} }=W_{1}^{*}(-x;1,\frac{\qnum{3}}{\qnum{2}}-2 \vert q),
\end{equation*}
and
\begin{equation*}
\lim_{q \uparrow 1} W_{1}^{*}(x;1,\frac{\qnum{3}}{\qnum{2}}-2 \vert q)=\frac{x^{2}}{\sqrt{1-x^{2}}},
\end{equation*}
i.e. the weight function of the fifth kind Chebyshev polynomials \cite{MR2270049}.

\subsubsection{Sixth kind $q$-Chebyshev polynomials}

The $q$-difference equation
\begin{multline}\label{eq:368}
x^{2}(1-x^2) D_{q} D_{q^{-1}} \phi_{n}(x;q)
+qx\left(  -\qnum{5} x^{2}+q+1 \right) D_{q} \phi_{n}(x;q) 
\\+\left(
\qnum{n} (q \qnum{5}- \qnum{1-n} ) x^{2} - q (q+1) \sigma_{n}
 \right) \phi_{n}(x;q) =0,
\end{multline}
is a special case of (\ref{eq:gup}) for $\alpha=1$ and $\beta=\qnum{5}/\qnum{2}-2$ with the polynomial solution 
\begin{multline*}
\phi_{n}(x;1,\frac{\qnum{5}}{\qnum{2}}-2 \vert q)=\qphin{m}{-q \qnum{5}}{q (q+1)}{-1}{1}{x}\\
=
x^{\sigma_{n}} 
\qhyper{2}{1}
{q^{\sigma_{n}-n},q^{n+\sigma_{n}-1} \left(q^6-q+1\right)}
{q^{2 \sigma_{n}+1}\left(q^3-q+1\right)}
{q^2}
{q^2 x^2},
\end{multline*}
satisfying the orthogonality relation of monic type
\begin{multline*}
\int_{-1}^{1} W_{1}^{*}(x;1, \frac{\qnum{5}}{\qnum{2}}-2 \vert q) \bar{\phi}_{n}(x;1, \frac{\qnum{5}}{\qnum{2}}-2 \vert q) \bar{\phi}_{m}(x;1, \frac{\qnum{5}}{\qnum{2}}-2 \vert q) d_{q}x \\
=d_{n,1,\frac{\qnum{5}}{\qnum{2}}-2}^{2} \left( \int_{-1}^{1} W_{1}^{*}(x;1, \frac{\qnum{5}}{\qnum{2}}-2 \vert q) d_{q}x,
\right) \delta_{n,m},
\end{multline*}
in which 
\begin{equation*}
W_{1}^{*}(x;1, \frac{\qnum{5}}{\qnum{2}}-2  \vert q)=\frac{\left(q^2 x^2;q^2\right){}_{\infty }
   \left(q^3-q+1\right)^{\frac{\log
   \left(x^2\right)}{2 \log
   (q)}}}{ \,\left(\frac{q^6-q+1
  }{q^3-q+1}  x^2;q^2\right){}_{\infty }}=W_{1}^{*}(-x;1,  \frac{\qnum{5}}{\qnum{2}}-2 \vert q),
\end{equation*}
and
\begin{equation*}
\lim_{q \uparrow 1} W_{1}^{*}(x;1, \frac{\qnum{5}}{\qnum{2}}-2  \vert q)=x^{2}\sqrt{1-x^{2}},
\end{equation*}
i.e. the weight function of the sixth kind Chebyshev polynomials \cite{MR2270049}.

\subsection{A generalization of $q$-Hermite polynomials}
Let us consider the $q$-dif\-ference equation
\begin{multline*}
x^{2}\left( \left(1-q^2\right) x^2-1 \right)D_{q} D_{q^{-1}} \phi_{n}(x;q)
+(q+1) x \left(x^2+p\right) D_{q} \phi_{n}(x;q) 
\\+\left(
q \qnum{-n} x^{2} -\sigma_{n} p  
 \right) \phi_{n}(x;q) =0,
\end{multline*}
as a special case of (\ref{eq:qde}) with the polynomial solution
\begin{multline}\label{eq:64n}
\phi_{n}(x;p \vert q)=\qphin{m}{1+q}{p(1+q)}{1-q^{2}}{-1}{x}\\
=x^{\sigma_{n}} \qhyper{2}{1}
{q^{\sigma_{n}-n},0}
{q^{2 \sigma_{n}+1} \left(-p  q^2+p +1\right)}
{q^2}{q^2 \left(1-q^2\right) x^2}.
\end{multline}

The polynomial sequence (\ref{eq:64n}) satisfies an orthogonality relation as
\begin{equation*}
\int_{-\alpha}^{\alpha} W_{2}^{*}(x;p \vert q) \phi_{n}(x;p \vert q) \phi_{m}(x;p;q) d_{q}x = 
\left( \int_{-\alpha}^{\alpha} W_{2}^{*}(x;p \vert q) \phi_{n}^{2}(x;p \vert q) d_{q}x \right) \delta_{n,m},
\end{equation*}
in which $\alpha=1/\sqrt{1-q^{2}}$, $W_{2}^{*}(x;p \vert q)=x^{2}W_{2}(x;p \vert q)$ is the main weight function and $W_{2}(x;p\vert q)$ satisfies the equation
\begin{equation}\label{eq:4460}
\frac{W_{2}(qx;p \vert q)}{W_{2}(x;p \vert q)}=\frac{-p  q^2+p +1}{q^2+\left(q^2-1\right) q^4 x^2}.
\end{equation}

Up to a periodic function, a solution of the equation (\ref{eq:4460}) is in the form
\begin{equation*}
W_{2}(x;p\vert q)=\frac{\left(p  \left(1-q^2\right)+1\right)^{\frac{\log
   \left(x^2\right)}{2 \log (q)}} 
   \qf{q^2 \left(1-q^2\right)
   x^2}{q^{2}}{\infty} }
{x^{2}}=W_{2}(-x;p \vert q),
\end{equation*}
where
\begin{equation*}
\lim_{q \uparrow 1} W_{2}^{*}(x;p \vert q)=\lim_{q \uparrow 1} x^{2} W_{2}(x;p \vert q)=x^{-2p}e^{-x^{2}},
\end{equation*}
appearing the weight function of generalized Hermite polynomials  \cite[Eq. (80)]{MR2421844}.

The monic type of polynomials (\ref{eq:64n}) satisfies a three term recurrence relation of type (\ref{eq:30}) with
\begin{equation*}
C_{2m,q}=-\frac{\left(p  \left(q^2-1\right)-1\right) q^{2 m-1}
   \left(q^{2 m}-1\right)}{q^2-1},
\end{equation*}
and
\begin{equation*}
C_{2m+1,q}=\frac{\left(-p  q^2+p +1\right) q^{4 m+1}-q^{2
   m}}{q^2-1},
\end{equation*}
such that
\begin{equation*}
\lim_{q \uparrow 1} C_{n,q}=\frac{1}{2} \left(p  \left((-1)^n-1\right)+n\right),
\end{equation*}
exactly coincides with \cite[Eq. (79.1)]{MR2270049}.

Consequently, the norm square value corresponding to the the monic type of polynomials (\ref{eq:64n}) takes the form
\begin{equation*}
\int_{-\alpha}^{\alpha} \bar{\phi}_{n}^{2}(x;p \vert q)W_{2}^{*}(x;p \vert q) d_{q}x
= d_{n,p}^{2}  \int_{-\alpha}^{\alpha} W_{2}^{*}(x;p \vert q)\, d_{q}x,
\end{equation*}
where
\begin{equation*}
d_{2m,p}^{2}=\frac{1}{2} q^{m (2 m-1)}  
\frac{\left(-p q^2+p +1\right)^m}{\left(q^2-1\right)^{2 m}} (-1;q)_{m+1} (q;q)_m \left(-p  q^3+p
    q+q;q^2\right)_m,
\end{equation*}
and
\begin{equation*}
d_{2m+1,p}^{2}=\frac{1}{2} q^{m (2 m+1)} 
  \frac{ \left(-p  q^2+p +1\right)^m}{\left(q^2-1\right)^{2 m+1}}  (-1;q)_{m+1} (q;q)_m
   \left(-p  q^3+p  q+q;q^2\right)_{m+1}.
\end{equation*}

Here we point out that if $p=0$, then the weight function of discrete $q$-Hermite I polynomials appears as
\begin{equation*}
W_{2}^{*}(x;0 \vert q)=\frac{1}{\left(\left(1-q^2\right) x^2;q^2\right){}_{\infty}},
\end{equation*}
and moreover
\begin{equation*}
\phi_{n}(x;0 \vert q)=k_{n} h_{n}(x\sqrt{1-q^{2}} \vert q),
\end{equation*}
in which $h_{n}(x \vert q)$ denotes the discrete $q$-Hermite I polynomials \cite[Eq. (14.28.1)]{MR2656096}, and $k_{n}$ is a normalizing constant. See also \cite{MR1860758,MR2027790} for further generalizations of $q$-Hermite polynomials.


\begin{thebibliography}{10}

\bibitem{MR2027790}
R.~{{\'A}}lvarez Nodarse, M.~K. Atakishiyeva, and N.~M. Atakishiyev.
\newblock A {$q$}-extension of the generalized {H}ermite polynomials with the
  continuous orthogonality property on {$\mathbb R$}.
\newblock {\em Int. J. Pure Appl. Math.}, 10(3):335--347, 2004.

\bibitem{AREAMASJED}
I.~Area and M.~Masjed-Jamei.
\newblock A symmetric generalization of {S}turm-{L}iouville problems in
  $q$-difference spaces.
\newblock submitted.

\bibitem{MR1810939}
G.~B. Arfken and H.~J. Weber.
\newblock {\em Mathematical methods for physicists}.
\newblock Harcourt/Academic Press, Burlington, MA, fifth edition, 2001.

\bibitem{MR1860758}
C. Berg and A. Ruffing.
\newblock Generalized {$q$}-{H}ermite polynomials.
\newblock {\em Comm. Math. Phys.}, 223(1):29--46, 2001.

\bibitem{MR0105586}
W.~E. Byerly.
\newblock {\em An elementary treatise on {F}ourier's series and spherical,
  cylindrical, and ellipsoidal harmonics, with applications to problems in
  mathematical physics}.
\newblock Dover Publications Inc., New York, 1959.

\bibitem{MR0481884}
T.~S. Chihara.
\newblock {\em An introduction to orthogonal polynomials}.
\newblock Gordon and Breach Science Publishers, New York, 1978.
\newblock Mathematics and its Applications, Vol. 13.

\bibitem{MR2128719}
G. Gasper and M. Rahman.
\newblock {\em Basic hypergeometric series}, volume~96 of {\em Encyclopedia of
  Mathematics and its Applications}.
\newblock Cambridge University Press, Cambridge, second edition, 2004.

\bibitem{MR2191786}
M. E.~H. Ismail.
\newblock {\em Classical and quantum orthogonal polynomials in one variable},
  volume~98 of {\em Encyclopedia of Mathematics and its Applications}.
\newblock Cambridge University Press, Cambridge, 2005.

\bibitem{JACKSON1910}
F.H. Jackson.
\newblock On $q$-definite integrals.
\newblock {\em Q. J. Pure Appl. Math.}, 41:193--203, 1910.

\bibitem{0817.39004}
A. Jirari.
\newblock {Second-order Sturm-Liouville difference equations and orthogonal polynomials.}
\newblock {\em Mem. Am. Math. Soc.}, 542:138 p., 1995.

\bibitem{MR2656096}
R. Koekoek, P-~A. Lesky, and R.~F. Swarttouw.
\newblock {\em Hypergeometric orthogonal polynomials and their {$q$}-analogues}.
\newblock Springer Monographs in Mathematics. Springer-Verlag, Berlin, 2010.

\bibitem{MR2270049}
M. Masjed-Jamei.
\newblock A basic class of symmetric orthogonal polynomials using the extended {S}turm-{L}iouville theorem for symmetric functions.
\newblock {\em J. Math. Anal. Appl.}, 325(2):753--775, 2007.

\bibitem{MR2374588}
M. Masjed-Jamei.
\newblock A generalization of classical symmetric orthogonal functions using a symmetric generalization of {S}turm-{L}iouville problems.
\newblock {\em Integral Transforms Spec. Funct.}, 18(11-12):871--883, 2007.

\bibitem{MR2421844}
M. Masjed-Jamei.
\newblock A basic class of symmetric orthogonal functions using the extended {S}turm-{L}iouville theorem for symmetric functions.
\newblock {\em J. Comput. Appl. Math.}, 216(1):128--143, 2008.

\bibitem{MR2601301}
M. Masjed-Jamei.
\newblock A basic class of symmetric orthogonal functions with six free parameters.
\newblock {\em J. Comput. Appl. Math.}, 234(1):283--296, 2010.

\bibitem{MASJEDAREA}
M. Masjed-Jamei and I. Area.
\newblock A symmetric generalization of {S}turm-{L}iouville problems in discrete spaces.
\newblock {\em Journal of Difference Equations and Applications}, 2013.

\bibitem{MASJEDAREA2}
M. Masjed-Jamei and I. Area.
\newblock A basic class of symmetric orthogonal polynomials of a discrete variable.
\newblock {\em J. Math. Anal. Appl.}, 399(291--305), 2013.

\bibitem{MR2467682}
M. Masjed-Jamei and M. Dehghan.
\newblock A generalization of {F}ourier trigonometric series.
\newblock {\em Comput. Math. Appl.}, 56(11):2941--2947, 2008.

\bibitem{MR2743534}
M. Masjed-Jamei and W. Koepf.
\newblock On incomplete symmetric orthogonal polynomials of {J}acobi type.
\newblock {\em Integral Transforms Spec. Funct.}, 21(9-10):655--662, 2010.

\bibitem{1220.33011}
M. Masjed-Jamei and W. Koepf.
\newblock {On incomplete symmetric orthogonal polynomials of Laguerre type.}
\newblock {\em Appl. Anal.}, 90(3-4):769--775, 2011.

\bibitem{MR1149380}
A.~F. Nikiforov, S.~K. Suslov, and V.~B. Uvarov.
\newblock {\em Classical orthogonal polynomials of a discrete variable}.
\newblock Springer Series in Computational Physics. Springer-Verlag, Berlin, 1991.

\bibitem{MR922041}
A.~F. Nikiforov and V.~B. Uvarov.
\newblock {\em Special functions of mathematical physics. A unified introduction with applications}.
\newblock Birkh{\"a}user Verlag, Basel, 1988.

\bibitem{Thomae1969}
J.~Thomae.
\newblock Beitrage zur theorie der durch die heinesche reihe.
\newblock {\em J. reine angew. Math}, 70:258--281, 1869.

\end{thebibliography}

\end{document}